\documentclass[12pt,a4paper]{article}
\usepackage{graphicx}
\usepackage{times,amsmath,theorem,latexsym,amssymb}

\advance\hoffset by -0.3cm

\makeatletter
\def\@maketitle{%
  \newpage
  {\vspace*{-11ex}%
  \small\noindent\begin{tabular}{@{}l}
  \end{tabular}}
  \null
  \vskip 5em%
  \begin{flushleft}%
    {\LARGE \bf\@title \par}
    \vskip 1.5em%
       {\large\bf \lineskip .5em \@author \par}%
    \vskip 1.0em%
  \end{flushleft}%
  \par
  \vskip 1.5em}
\renewcommand\section{\@startsection {section}{1}{\z@}%
                                   {-3.5ex \@plus -1ex \@minus -.2ex}%
                                   {1.5ex \@plus.2ex}%
                                   {\normalfont\Large\bfseries}}
\renewcommand\subsection{\@startsection{subsection}{2}{\z@}%
                                     {-3.0ex\@plus -2ex \@minus -.2ex}%
                                     {1.0ex \@plus .2ex}%
                                     {\normalfont\large\bfseries}}
\renewcommand\subsubsection{\@startsection{subsubsection}{3}{\z@}%
                                     {-2.5ex\@plus -2ex \@minus -.2ex}%
                                     {0.9ex \@plus .2ex}%
                                     {\normalfont\normalsize\bfseries}}
\renewcommand\paragraph{\@startsection{paragraph}{4}{\z@}%
                                     {-2.0ex \@plus1ex \@minus.2ex}%
                                     {-1em}%
                                     {\normalfont\normalsize\bfseries}}
\renewcommand\subparagraph{\@startsection{subparagraph}{5}{\parindent}%
                                       {-1.5ex \@plus1ex \@minus .2ex}%
                                       {-1em}%
                                      {\normalfont\normalsize\bfseries}}
\def\@evenhead{\footnotesize\thepage\hfil\slshape\leftmark}%
\def\@oddhead{\footnotesize{\slshape\rightmark}\hfil\thepage}%
\numberwithin{equation}{section} \numberwithin{figure}{section}
\numberwithin{table}{section}

\renewcommand{\@makecaption}[2]{\begin{quote}
\footnotesize {\bf #1}~#2
\end{quote}}
\makeatother
{\vskip\baselineskip}
\newenvironment{proof}{{\vskip\baselineskip\noindent\textbf{Proof.}}}%
{\hspace*{.1pt}\hspace*{\fill}\BOX\vskip\baselineskip}

\newcommand{\BOX}{\ensuremath\Box}
\newtheorem{theorem}{Theorem}[section]

{\theorembodyfont{\rmfamily}}
{\theorembodyfont{\rmfamily}}
{\theorembodyfont{\rmfamily}}
\newtheorem{lemma}[theorem]{Lemma}
{\theorembodyfont{\rmfamily}}
{\theorembodyfont{\rmfamily}}
{\vskip\baselineskip\noindent\textbf{Proof of {#1}.}}%
{\hspace*{.1pt}\hspace*{\fill}\BOX\vskip\baselineskip}
{\vskip\baselineskip\noindent\textbf{Proof of Theorem \protect\ref{#1}:}}%
{\hspace*{.1pt}\hspace*{\fill}\BOX\vskip\baselineskip}
{\vskip\baselineskip\noindent\textbf{Proof of Theorems \protect\ref{#1} --
\protect\ref{#2}:}}%
{\hspace*{.1pt}\hspace*{\fill}\BOX\vskip\baselineskip}
\usepackage{amsmath,amssymb,amsfonts,mathrsfs,latexsym}

\newcommand{\C}{\mathbb{C}}
\newcommand{\N}{\mathbb{N}}

\newcommand{\Z}{\mathbb{Z}}

\newcommand{\R}{\mathbb{R}}

\usepackage{mathrsfs}

\usepackage{txfonts} 

\usepackage{srcltx}

\numberwithin{equation}{section}

\begin{document}
\thispagestyle{myheadings}
\markboth{Gawronski et al.}{Asymptotics of Chebyshev-Stirling and Stirling Numbers of the Second Kind}

\title{\bf 
\begin{center}
Asymptotics of Chebyshev-Stirling and Stirling numbers of the second kind
\end{center}
}
\author{\bf \centerline{Wolfgang Gawronski $^{a}$, Lance L. Littlejohn $^{b, *}$, Thorsten Neuschel $^{c, **}$}}
\maketitle
\date{\today}

\vspace*{0.5cm}

{\bf Abstract:} For the Chebyshev-Stirling numbers, a special case of the Jacobi-Stirling numbers, asymptotic formulae are derived in terms of a local central limit theorem. The underlying probabilistic approach also applies to the classical Stirling numbers of the second kind. Thereby a supplement of the asymptotic analysis for these numbers is established.\\



\section{Introduction and summary} \label{intro}
The first objects of our consideration are the celebrated Stirling numbers of the second kind or partition numbers. Following Knuth \cite{21}, see also \cite{14}, we denote these numbers by the curly bracket symbol
$\big\{ ^n_j\big\}, n,j$ being non-negative integers. It is well-known \cite{9}, \cite{14} that they satisfy the triangular recurrence relation
\begin{align} \label{eq:1.1}
\left\{ \begin{array}{c}
n \\ j  \end{array}
\right\} =
\left\{    \begin{array}{c}
n-1 \\ j-1 \end{array}
\right\} + j 
\left\{  \begin{array}{c}
n-1 \\ j \end{array}
\right\}, ~
\left\{ \begin{array}{c}
n \\ 0  \end{array}
\right\} = \delta_{n,0}\, , ~
\left\{ \begin{array}{c}
0 \\ j  \end{array}
\right\} = \delta_{0,j}\, ,
\end{align}
$\delta_{n,j}$ being Kronecker's symbol, and they admit the explicit representation
\begin{align} \label{eq:1.2}
\left\{ \begin{array}{c}
n \\ j  \end{array}
\right\}
= \frac{1}{j!} \sum_{r=0}^j {j \choose r}\, (-1)^r \,(j-r)^n.
\end{align}
Our second objects are the \textit{Jacobi-Stirling numbers of the second kind} which we denote by
$\big\{ ^n_j\big\}_\gamma$, where $\gamma$ is a fixed and positive parameter and $n,j$ are non-negative integers. These numbers were discovered in the study of the spectral theory of powers of the Jacobi differential operator \cite{12}, \cite{13}, \cite{24}. A formal definition of these numbers can be given through the triangular recurrence relation
\begin{align} \label{eq:1.3}
\left\{ \begin{array}{c}
n \\ j  \end{array}
\right\}_\gamma = 
\left\{ \begin{array}{c}
n-1 \\ j-1 \end{array}
\right\}_\gamma + j (j + 2\gamma - 1)
\left\{ \begin{array}{c}
n-1 \\ j \end{array}
\right\}_\gamma , ~
\left\{ \begin{array}{c}
n \\ 0 \end{array} 
\right\}_\gamma = \delta_{n,0}, ~
\left\{ \begin{array}{c}
0 \\ j  \end{array}
\right\}_\gamma = \delta_{0,j}\, ,
\end{align}
\cite{3}, \cite{17}. During the past decade they received considerable attention resulting in a series of papers on differential equations, combinatorics, and graph theory \cite{1}, \cite{2}, \cite{3},\cite{7}, \cite{8}, \cite{10}, \cite{12}, \cite{13}, \cite{17}, \cite{18}, \cite{24}, \cite{26}, \cite{27}. A brief account of the background is given in section 2 below. \\

For the unique solution of the recurrence \eqref{eq:1.3} the following representation is known, \cite{3}, \cite{13}, \cite{17},
\begin{equation} \label{eq:1.4}
\left\{ \begin{array}{c}
n \\ j  \end{array}
\right\}_\gamma = \sum_{r=0}^j (-1)^{r+j}\,
\frac{(2r+2\gamma-1) \Gamma (r+2\gamma-1) \big( r(r+2\gamma-1)\big)^n}{r! (j-r)! \Gamma (j+r+2\gamma)}\, ,
\end{equation}
where $n,j \in \N_0 := \{0,1,2,\ldots\}$ and $\gamma > 0$. We immediately infer the easy asymptotic statement
\begin{equation} \label{eq:1.5}
\left\{ \begin{array}{c}
n \\ j  \end{array}
\right\}_\gamma \sim \frac{\Gamma (j+2\gamma-1)}{j! \Gamma (2j+2\gamma-1)} \,
\big( j(j+2\gamma-1)\big)^n, \qquad\text{as } n\to\infty\, ,
\end{equation}
provided that $j \geq 1$ is \textit{fixed}, where throughout the symbol $\sim$ means that the ratio of both sides in \eqref{eq:1.5} tends to 1 as usual. This property corresponds to the well known trivial asymptotics for the classical Stirling numbers
\begin{equation} \label{eq:1.6}
\left\{ \begin{array}{c}
n \\ j  \end{array}
\right\} \sim \frac{j^n}{j!}, \qquad \text{as } n\to\infty, ~ j \textit{ fixed},
\end{equation}
resulting from \eqref{eq:1.2}. Much more interesting are asymptotics for these numbers holding \textit{uniformly} with respect to $j$. The literature on Stirling numbers includes a variety of such results for $\big\{ ^n_j\big\}$, however all of them contain quantities which are undetermined or they are defined implicitly as solutions of certain transcendental equations only, see e. g. \cite{4}, \cite{5}, \cite{9}, \cite{14}, \cite{28}, \cite{36}.  For the Jacobi-Stirling numbers $\big\{ ^n_j\big\}_\gamma$ besides the ``pointwise asymptotics'' in \eqref{eq:1.5} none is known as far as the authors are aware. So it is natural to ask for asymptotics regarding the sequences $\big\{ ^n_j\big\}_\gamma$ and $\big\{ ^n_j\big\}$ as $n\to\infty$, both holding uniformly with respect to $j\in\Z$, where $\Z$ denotes the set of integers as customary. In particular except for error terms the leading parts should be given \textit{explicitly} and in terms of \textit{elementary} functions including factorials.\\

For technical reasons which will be explained in section 3 we have to restrict the Jacobi-Stirling case to the special value $\gamma = 1/2$. Regarding the origin of the numbers $\big\{ ^n_j\big\}_{1/2}$ mentioned above, this case corresponds to the Chebyshev differential operator. Hence in the sequel we call these numbers \textit{Chebyshev-Stirling} numbers of the second kind. For details, see section 2. To be more precise the two main results of this article give affirmative answers to the above raised questions by the following asymptotic forms in terms of local central limit theorems. We prove
\begin{equation} \label{eq:1.7}
\frac{2\sqrt{b_n}\, j! (\log 2)^{n+1}}{n!} \, 
\left\{ \begin{array}{c}
n \\ j  \end{array}
\right\} = \frac{1}{\sqrt{2\pi}}\, e^{-x^2/2} 
\left( 1 + \frac{c_n (x^3 - 3x)}{6\sqrt{n}}\right) + o \left( \frac{1}{\sqrt{n}}\right)\, ,
\end{equation}
as $n\to\infty$, where $x = (j-a_n)/\sqrt{b_n}$, and elementary expressions $a_n, b_n, c_n$ given explicitly in Lemma 5.2 below. Similarly we get
\begin{equation} \label{eq:1.8}
\frac{\sqrt{5b_n'} (2j)! \omega^{2n+1}}{2(2n)!} \, 
\left\{ \begin{array}{c}
n \\ j  \end{array} \right\}_{1/2} =
\frac{1}{\sqrt{2\pi}}\, e^{-x^2/2}\, 
\left( 1 + \frac{c_n' (x^3 - 3x)}{6\sqrt{n}}\right) + o \left( \frac{1}{\sqrt{n}}\right)\, ,
\end{equation}
as $n\to\infty$, where $x = (j-a_n')/\sqrt{b_n'}$ and again elementary quantities $a_n', b_n', c_n', \omega$ given explicitly in Lemma 6.1. \\

For a discussion of the approximation of the kind \eqref{eq:1.7}, \eqref{eq:1.8} we refer to sections 4, 5, 6. In particular both remainder terms hold uniformly with respect to $j\in\Z$. In establishing \eqref{eq:1.7} and \eqref{eq:1.8}, the basic tools are taken from the central limit theory of probability and from special functions by the so-called Lindel\"of-Wirtinger expansion for the  Euler-Frobenius polynomials (Lemma 3.1)
. This approach to asymptotics occasionally has been applied to various special functions and sequences in the literature, see e. g.
\cite{4}, \cite{6}, \cite{9}, \cite{14}, \cite{16}, \cite{19}, \cite{35}, \cite[Chapter 3]{37}.
\section{Background of the Chebyshev-Stirling numbers}
Most of the current interest and research in the Jacobi-Stirling numbers of the second kind
$%
\genfrac{\{}{\}}{0pt}{}{n}{j}%
_{\gamma}$ lies in the applications to combinatorics and graph theory; the
application of these numbers in these areas is natural since they
mimic many of the properties of the classical Stirling numbers of the second kind.\\ 

It is important to note, however that these numbers were originally discovered
in the left-definite spectral analysis of integral powers of the second-order
Jacobi differential expression%
\[
\ell\lbrack y](x):=\frac{1}{(1-x)^{\alpha}(1+x)^{\beta}}\left(  -\left(
(1-x)^{\alpha+1}(1+x)^{\beta+1}y^{\prime}(x)\right)  ^{\prime}+k(1-x)^{\alpha
}(1+x)^{\beta}\right)  ,
\]
where, for classical reasons, the constants $\alpha,\beta>-1$ and $k\geq0$ are
fixed. Indeed, in \cite{13}, the authors proved that, for each
$n\in\mathbb{N},$
\begin{equation}
\ell^{n}[y](x)=\dfrac{1}{(1-x)^{\alpha}(1+x)^{\beta}}\sum_{j=0}^{n}%
(-1)^{j}\left(  c_{j}^{\gamma}(n,k)(1-x)^{\alpha+j}(1+x)^{\beta+j}%
y^{(j)}(x)\right)  ^{(j)}\label{Jacobi powers special case}%
\end{equation}
for $x\in(-1,1),$ where $2\gamma -1 = \alpha + \beta +1$, and where the coefficients $c_{j}^{\gamma}(n,k)$ $(j=0,1,\ldots
n)$ are non-negative and given explicitly by
\[
c_{0}^{\gamma}(n,k)=\left\{
\begin{array}
[c]{ll}%
0 & \text{if }k=0\\
k^{n} & \text{if }k>0
\end{array}
\right.  \text{ and, for $j \ge 1$,  }c_{j}^{\gamma}(n,k)=\left\{
\begin{array}
[c]{ll}%
\genfrac{\{}{\}}{0pt}{}{n}{j}%
_{\gamma} & \qquad{\text{if }k=0}\\
\sum_{r=0}^{n-j}\binom{n}{r}%
\genfrac{\{}{\}}{0pt}{}{n-r}{j}%
_{\gamma}k^{r} & \qquad{\text{if }k>0}.
\end{array}
\right.
\]
The Chebyshev differential operator corresponds to the case $\alpha = \beta = -1/2$, that is $\gamma = 1/2$. Details of
these analytic results can be found in \cite{13}.\\

Starting with the Legendre-Stirling case $\gamma =1$ in \cite{1}, which corresponds to the parameters $\alpha = \beta = 0$, in the sequel a series of authors developed various combinatorial models, thereby illustrating the significance of the Jacobi-Stirling numbers and related quantities. We refer to the articles \cite{2}, \cite{3}, \cite{7}, \cite{8}, \cite{10}, \cite{17}, \cite{18}, \cite{26}, \cite{27}. However none of these papers contains information on asymptotics.

\section{Elementary properties and analytic tools}
In this section we collect some analytic facts which are relevant for proving our main results below. Along with these preparations we already obtain new information for the sequences under consideration. \\

To begin with we introduce a special case of the Euler-Frobenius polynomials, $P_n$ say, which usually are introduced via the rational function expansion
\begin{equation} \label{eq:3.1}
\sum_{\nu=0}^\infty \nu^n z^\nu = \left( z\, \frac{d}{dz}\right)^n \, \frac{1}{1-z} =
\frac{P_n (z)}{(1-z)^{n+1}}\, , \qquad n\in\N_0\, .
\end{equation} 
The power series and the polynomials $P_n$ of degree $n$ play important roles in various parts of mathematics and many of its applications as described for instance in \cite{9}, \cite{14}, \cite{15}, \cite{16}, \cite{22}, \cite{31}, \cite{32}. Here we only present a few properties which are significant for our purpose.
\begin{lemma} \label{lem3.1}
Suppose that $n$ is a positive integer.
\begin{itemize}
\item[i)]  All zeros of $P_n$ are real, simple, non-positive, $n$ in number, and $P_n (0) = 0$; moreover
$P_n (-1) = 0$ iff $n$ is even. If $n = 2k$ (respectively  $2k+1$), then $k-1$ (respectively $k$) zeros are located  in each of the intervals
$(-\infty,-1)$ and $(-1,0)$.
\item[ii)] For $z\in\C\setminus \{1\}$ we have
\begin{equation} \label{eq:3.2}
\frac{P_n (z)}{(1-z)^{n+1}} = n! \sum_{m=-\infty}^\infty\, \frac{1}{\big(2\pi im + \log (1/z)\big)^{n+1}}\, .
\end{equation}
\end{itemize}
\end{lemma}
Regarding proofs of the properties for the zeros we refer to \cite{22}, \cite{30}. The series in \eqref{eq:3.2}  is a special case of the so-called \textit{Lindel\"of-Wirtinger expansion} \cite{16}, \cite{23}, \cite[p. 34]{25}, \cite{38}. For $\log (1/z)$ we may choose that branch with $\text{\rm Im} \log (1/z)\in [0,2\pi)$. Then the analyticity on the punctured cut $(0,\infty)\setminus \{1\}$ is generated by the summation over all branches of the logarithm. At $z = 0$, the right-hand side is defined by continuity. Occasionally we write the Lindel\"of-Wirtinger expansion by
\begin{equation} \label{eq:3.3}
\frac{P_n (e^{-w})}{(1-e^{-w})^{n+1}} = n! \sum_{m=-\infty}^\infty\, \frac{1}{(w+2\pi im)^{n+1}}
\end{equation}
using an \textit{Eisenstein} series which is convenient for computing derivatives \cite{11}, \cite[p. 234]{33}.\\

An asymptotic analysis for many sequences having a combinatorial meaning often depends on the explicit knowledge of corresponding generating functions. With this in mind and on the basis of the representation formula \eqref{eq:1.4} for the Jacobi-Stirling numbers it turns out that in the special case $\gamma = 1/2$ we are able to obtain suitable identities which finally lead to asymptotic forms for the Chebyshev-Stirling numbers $\big\{ ^n_j\big\}_{1/2}$ having the properties described in the introduction (see section 6). A basic representation formula for the Chebyshev-Stirling numbers is given by the following analogue of \eqref{eq:1.2} for the Stirling numbers.
\begin{lemma} \label{lem3.2}
If $n,j\in\N_0$, then we have 
\begin{equation} \label{eq:3.4}
\left\{ \begin{array}{c}
n \\ j  \end{array}
\right\}_{1/2} = \frac{1}{(2j)!}\sum_{r=0}^{2j} \, {2j \choose r} \, (-1)^r \, (j-r)^{2n}.
\end{equation}
\end{lemma}
\begin{proof}
Putting $\gamma = 1/2$ in \eqref{eq:1.4} we obtain
\[
\left\{ \begin{array}{c}
n \\ j  \end{array}
\right\}_{1/2} = 2 \sum_{r=0}^j (-1)^{r+j} \, \frac{r^{2n}}{(j-r)! (j+r)!}
\]
and further
\[
\left\{ \begin{array}{c}
n \\ j  \end{array}
\right\}_{1/2} = \frac{1}{(2j)!} \, \sum_{r=0}^j (-1)^{r+j}\,  {2j \choose j-r}\, r^{2n} +
\frac{1}{(2j)!} \, \sum_{r=0}^j (-1)^{r+j} \, {2j \choose j+r}\, r^{2n}.
\]
After making the index shifts $j - r \to r$ and $j + r \to r$, we obtain \eqref{eq:3.4}.
\end{proof}
In \cite[(3.4)]{3} the authors derived a formula connecting the Legendre-Stirling numbers $\big\{ ^n_j\big\}_1$ with the classical Stirling numbers $\big\{ ^n_j\big\}$. As a supplement we relate $\big\{ ^n_j\big\}_{1/2}$ to $\big\{ ^n_j\big\}$. To this end we start from \eqref{eq:3.4} in Lemma 3.2 giving
\begin{align*}
\left\{ \begin{array}{c}
n \\ j  \end{array}
\right\}_{1/2}
& = \frac{1}{(2j)!} \, \sum_{r=0}^{2j}\, {2j \choose r} \, (-1)^r\, \left( \frac{\partial}{\partial x}\right)^{2n}
    e^{(j-r)x}\Bigg|_{x = 0} \\*[0.3cm]
& = \frac{1}{(2j)!}\, \left( \frac{\partial}{\partial x}\right)^{2n} e^{jx} \, (1 - e^{-x})^{2j}\Bigg|_{x=0} \\*[0.3cm]
& = \left( \frac{\partial}{\partial x}\right)^{2n} e^{-jx} \, \frac{(e^x-1)^{2j}}{(2j)!}\Bigg|_{x=0} =
    \left( \frac{\partial}{\partial x}\right)^{2n} e^{-jx} \,\Phi_{2j} (x) \Bigg|_{x=0} \, ,
\end{align*}
where $\Phi_{j} (t) := \sum_{n=0}^\infty \, \frac{\big\{ ^n_j\big\}}{n!}\, t^n = \frac{(e^t-1)^j}{j!}$ denotes the vertical generating function of the Stirling numbers of the second kind \cite[p. 206]{9}. Continuing we find
\[
\left\{ \begin{array}{c}
n \\ j  \end{array} \right\}_{1/2} =
\sum_{\nu=0}^{2n} \, {2n\choose \nu} \, (-j)^{2n-\nu} \, \Phi_{2j}^{(\nu)} (0) =
\sum_{\nu=0}^{2n} \, {2n\choose \nu} \, (-j)^{2n-\nu} \,
\left\{   \begin{array}{c}
\nu \\ 2j \end{array} \right\}
\]
and
\begin{align}
    \left\{ \begin{array}{c} n \\ 0 \end{array} \right\}_{1/2}
& = \left\{ \begin{array}{c} 2n \\ 0 \end{array} \right\} = \delta_{n,0}\, ,  ~~
    \left\{ \begin{array}{c} n \\ j \end{array} \right\}_{1/2} 
 = j^{2n} \Delta_x^{2n}
    \frac{\left\{\begin{array}{c} x \\2j \end{array} \right\}}{j^x}\Bigg|_{x=0}, \quad \text{when } j\geq 1,
\end{align}
with an obvious meaning of the forward difference operator $\Delta_x$ acting on $x$. \\ \

Next we consider the modified Chebyshev-Stirling numbers $(2j)! \big\{ ^n_j\big\}_{1/2}$ and its horizontal generating function
\begin{equation} \label{eq:3.6}
L_n (s):= \sum_{j=0}^n (2j)! 
\left\{ \begin{array}{c}
n \\ j  \end{array} 
\right\}_{1/2} s^j, \qquad s\in\C,
\end{equation}
which for our purpose is appropriate rather than the generating function for the numbers $\big\{ ^n_j\big\}_{1/2}$. In case of the Stirling numbers see \cite[p. 108, example 5.4, c]{4} and Lemma 3.6 below. Generating functions, connection formulae with the Euler-Frobenius polynomials and the Lindel\"of-Wirtinger expansion are given by
\begin{lemma} \label{lem3.3} 
We have
\begin{flalign}
\label{eq:3.7} i)   & ~~ \sum_{n=0}^\infty \left\{ \begin{array}{c} n \\ j \end{array} \right\}_{1/2}\,
          \frac{t^{2n}}{(2n)!} = \frac{2^{2j}}{(2j)!}\, 
               \left( \sinh \frac{t}{2}\right)^{2j}, \qquad j\in \N_0 ,   \\ 
\label{eq:3.8} ii)  & ~~\sum_{n=0}^\infty L_n (s) \frac{t^{2n}}{(2n)!} = 
         \frac{1}{1-4s \sinh^2 (t/2)}\, ,   \\
\label{eq:3.9} iii) & ~~L_n \left( \frac{1}{2(\cosh w-1)}\right) = (2n)! \, \frac{2(\cosh w-1)}{\sinh w}\,
               \sum_{m=-\infty}^\infty\, \frac{1}{(w+2\pi im)^{2n+1}}\, ,  \\
\label{eq:3.10} iv)  & ~~ L_n \left( \frac{z}{(1-z)^2}\right) = \frac{2}{1+z}\, \frac{P_{2n} (z)}{(1-z)^{2n}}\, ,
\end{flalign} 
with obvious restrictions for the complex variables $s,t,w,z$. 
\end{lemma}
Regarding (3.7) it should be mentioned that the Chebyshev-Stirling numbers are related to the central factorial numbers of odd and even indices for which various generating functions have been computed in the literature, see for instance
\cite[section 4]{17}, \cite[chapter 6.5]{34}.
\begin{proof}
Part i) readily follows from Lemma 3.2 which implies ii) directly. The Lindel\"of-Wirtinger expansion in iii) results from ii) by putting $s = 1/2(\cosh w - 1)$ and a routine application of Cauchy's formula for the coefficients of power series in combination with residue calculus. Finally we obtain iv) from iii) and \eqref{eq:3.3} by observing that $z = e^{-w}$.
\end{proof}
Further, using elementary conformal mapping properties of the quadratic transformation \linebreak[4] $s = z/(1-z)^2$ and the connection formula \eqref{eq:3.10} from Lemma 3.1, i) we conclude
\begin{lemma} \label{lem3.4}
All zeros of $L_n$ are real, simple and they are located in the interval $\big( - \frac{1}{4},0\big]$.
\end{lemma}
Next, we assume that the reader is familiar with the concept of unimodality of sequences, for example \cite[section 7.1]{9}. Using Lemma \ref{lem3.4} in combination with a well-known criterion for unimodality, for example \cite[Theorem B, p. 270]{9}, we immediately obtain a supplement of the unimodality property of the Jacobi-Stirling numbers \cite[Theorem 5.4]{3}.
\begin{theorem} \label{theo3.5}
If $n\geq 3$, the numbers $(2j!)\, \big\{ ^n_j\big\}_{1/2}$, $0 \leq j \leq n$, are unimodal with either a peak or a plateau of two points.
\end{theorem}
Finally, according to the Chebyshev-Stirling numbers above and following Bender \cite[p. 108]{4} for the modified Stirling numbers $j!\,\big\{ ^n_j\big\}$, we consider its horizontal generating function
\begin{equation} \label{eq:3.11}
Q_n (s) := \sum_{j=0}^n j! \, \left\{ \begin{array}{c} n \\ j \end{array} \right \}\, s^j, \qquad s\in\C.
\end{equation}
\begin{lemma} \label{lem3.6}
We have
\begin{flalign}
\label{eq:3.12} i)   & ~~\sum_{n=0}^\infty Q_n (s) \frac{t^n}{n!} = \frac{1}{1-s(e^t-1)}\, , \\
\label{eq:3.13} ii)  & ~~Q_n \left( \frac{1}{e^w-1}\right) = n!\, \big( 1-e^{-w}\big) \sum_{m=-\infty}^\infty\,
             \frac{1}{(w+2\pi im)^{n+1}}\, , \\
\label{eq:3.14} iii) & ~~Q_n\left( \frac{z}{1-z}\right) = \frac{P_n (z)}{(1-z)^n}\, ,                         
\end{flalign}
with obvious restrictions for the complex variables $s,t,w,z$.
\end{lemma}
Part i) has been taken from Bender's classic paper \cite[p. 108]{4} and the reasonings for formula (3.13), (3.14) are very similar to those for (3.9), (3.10); therefore we omit the details. Further, again on the basis of Lemma \ref{lem3.1}, i) and the connection formula (3.14) we may infer
\begin{lemma} \label{lem3.7}
All zeros of $Q_n$ are real, simple and they are located in the interval $(-1,0]$.
\end{lemma}
Concluding this section we mention the unimodality of the sequence $j!\, \big\{ ^n_j\big\}$, $0 \leq j \leq n$, which already is due to Bender \cite[p. 309]{4}. This also is a consequence of Lemma \ref{lem3.7} by the same argument applied above.
\section{Probabilistic tools from central limit theory}
The probabilistic point of view provides one of numerous methods for computing  asymptotics for sequences of positive numbers
\cite{4}, \cite{6}, \cite{9}, \cite{14}, \cite{16}, \cite{19}, \cite{35}, \cite{37}. In this section we prepare the relevant terminology and give a general auxiliary result. \\

Having in mind the two main objects to be treated below, given by the generating polynomials $L_n$ in \eqref{eq:3.6} and $Q_n$ in (3.11), we consider polynomials
\begin{equation} \label{eq:4.1}
A_n (z) := \sum_{j=0}^n \alpha_{nj} z^j = \alpha_{nn} \prod_{\nu=1}^n (z+x_{n\nu})
\end{equation}
having real and non-positive zeros only, that is $x_{n\nu} \geq 0$, $\nu = 1,\ldots,n$. We ask for asymptotics of the coefficients 
$\alpha_{nj}$, as $n \to\infty$, uniformly in $j$. Rewriting \eqref{eq:4.1} as
\begin{equation} \label{eq:4.2}
\frac{A_n (z)}{A_n (1)} = \prod_{\nu=1}^n (p_{n\nu} z + 1 - p_{n\nu}),
\end{equation}
where $p_{n\nu} := 1/(1+x_{n\nu})$, the polynomials in \eqref{eq:4.2} may be regarded as the generating function of the row sums of a triangular array of Bernoulli random variables
\begin{equation} \label{eq:4.3}
\begin{array}{ccccc}
X_{11} \\
X_{21} & X_{22} \\
\vdots &        & \ddots \\
X_{n1} & X_{n2} & \ldots & X_{nn} \\
\vdots & \vdots &        &        & \ddots
\end{array}
\end{equation}
Here, due to the factorization in \eqref{eq:4.2}, the entries are row-wise independent with distributions given by
\begin{equation} \label{eq:4.4}
P(X_{n\nu} = 1) = p_{n\nu}, \quad P(X_{n\nu} = 0) = 1 - p_{n\nu}
\end{equation}
with numbers $p_{n\nu} \in [0,1]$. \\ 

In \cite{16} an asymptotic expansion for
\begin{equation} \label{eq:4.5}
p(n,j) := P(S_n = j)
\end{equation}
has been derived, where
\begin{equation} \label{eq:4.6}
S_n := \sum_{\nu=1}^n X_{n\nu}
\end{equation}
denote the row sums in the scheme \eqref{eq:4.3}. This has been done by a modification of a standard local central limit theorem for ``simple sums'' $\sum_{\nu=1}^n X_\nu$, where each component $X_\nu$ has a lattice distribution \cite{29}. Prior to a formulation we need some notation and conditions for the Bernoulli scheme \eqref{eq:4.3} with \eqref{eq:4.4} - \eqref{eq:4.6}. Let $\mu_n := E(S_n)$ and $\sigma_n^2 := \text{Var }(S_n)$ be the expectation and variance of $S_n$ respectively for which we assume that
\begin{equation} \label{eq:4.7}
\liminf_{n\to\infty} \frac{\sigma_n^2}{n} > 0.
\end{equation}
Further, suppose that the normalized cumulants of $S_n$ are defined by
\begin{equation} \label{eq:4.8}
\lambda_{\nu,n} := \frac{n^{(\nu-2)/2}}{\sigma_n^\nu} \, \frac{1}{i^\nu} 
\left( \frac{d}{dt}\right)^\nu \log E(e^{itS_n})\Big|_{t=0}\, ,
\end{equation}
$n,\nu\in\N, \nu\geq 2$, where $E(e^{itS_n})$ is the characteristic function of $S_n$ and $\log$ is that branch of the logarithm on the cut plane $\C\setminus (-\infty,0]$ satisfying $\log 1 = 0$. Finally we introduce the functions
\begin{equation} \label{eq:4.9}
q_{\nu,n} (x) := \frac{1}{\sqrt{2\pi}}~ e^{-x^2/2} 
\sum_{\mu_1 + 2\mu_2 + \ldots + \nu\mu_\nu = \nu} H_{\nu+2s} (x) \prod_{m=1}^\nu \, \frac{1}{\mu_m!}\,
\left( \frac{\lambda_{m+2,n}}{(m+2)!}\right)^{\mu_m},
\end{equation}
$x\in \R, n,\nu\in\N$, where $s = \mu_1 + \ldots + \mu_\nu$, and the modified Hermite polynomials are defined by
\begin{equation} \label{eq:4.10}
H_m (x) := (-1)^m e^{x^2/2} \left( \frac{d}{dx}\right)^m e^{-x^2/2},
\end{equation}
$x\in\R, m\in\N_0$. Now from \cite[Lemma 3.1]{16} we take our basic auxiliary result given by the following local central limit theorem for triangular Bernoulli arrays.
\begin{lemma} \label{lem4.1}
Assuming the above notations and the condition \eqref{eq:4.7}, then for every $k \geq 2$, we have
\begin{equation} \label{eq:4.11}
\sigma_n p(n,j) = \frac{1}{\sqrt{2\pi}}~ e^{-x^2/2} + \sum_{\nu=1}^{k-2}\, \frac{q_{\nu,n} (x)}{n^{\nu/2}} +
o \left( \frac{1}{n^{(k-2)/2}}\right),
\end{equation}
as $n\to\infty$, uniformly in $j \in \Z$, where $x = (j - \mu_n)/\sigma_n$.
\end{lemma}
The reader who is not familiar with expansions of type \eqref{eq:4.11} is refered to the comments accompanying Lemma \ref{lem3.1} of \cite{16} in general. Particular aspects regarding our applications of Lemma \ref{lem4.1} to the Stirling and the Chebyshev-Stirling numbers are discussed in the following two sections.
\section{Asymptotics for the Stirling numbers}
To begin with we mention the asymptotic normality of the surjection numbers, that is
\begin{equation} \label{eq:5.1}
\lim_{n\to\infty}\, \sup_{x\in\R} \Bigg| \frac{1}{Q_n (1)} \sum_{j\leq \big( n+x\sqrt{1-\log 2}\big)/2\log 2} 
j! \left\{ \begin{array}{c} n \\ j \end{array} \right\} - \frac{1}{\sqrt{2\pi}} \int\limits_{-\infty}^x
e^{-t^2/2} dt\Bigg| = 0,
\end{equation}
which appeared in \cite{4}, \cite[p. 654]{14}. In what follows we prove a corresponding local version by applying the preliminary results of the previous sections with $A_n (s) = Q_n (s)$ and
\begin{equation} \label{eq:5.2}
p(n,j) = \frac{j! \big\{ ^n_j\big\}}{Q_n (1)},
\end{equation}
where $Q_n (s)$ is given by (3.11). This is possible in view of Lemma \ref{lem3.7}. \\

First we compute moments and cumulants for the distribution in \eqref{eq:5.2} where we use the notations introduced in section 4. For that purpose it is convenient to have suitable estimates for the Eisenstein series occuring in the Lindel\"of-Wirtinger expansion \eqref{eq:3.3}.
\begin{lemma} \label{lem5.1}
If $n\in\N, n\geq 2$, and $w\in \R\setminus \{0\}$, then we have
\begin{equation}
\sum_{m=-\infty}^{\infty}\,\frac{1}{(w+2\pi im)^{n+1}}=\frac{1}{w^{n+1}%
}\left(  1+{\mathscr O}\left(  \frac{w}{4\pi}\right)  ^{(n+1)/2}\right)
,\label{eq:5.3}%
\end{equation}
the ${\mathscr O}$-term holding uniformly with respect to $w\in\R$.
\end{lemma}
\begin{proof}
Writing
\begin{equation*}
\sum_{m=-\infty}^\infty \, \frac{1}{(w+2\pi im)^{n+1}} = \frac{1}{w^{n+1}}\,
\left( 1 + \sum_{m\not= 0} \left( \frac{w}{w+2\pi im}\right)^{n+1}\right) =:
\frac{1}{w^{n+1}}\, \big( 1 + R_n (w)\big),
\end{equation*}
say, we have

\begin{equation*}
\big| R_n(w)\big| \leq 2 \sum_{m=1}^\infty \, \frac{|w|^{n+1}}{(w^2+4\pi^2m^2)^{(n+1)/2}} \leq
2\, \zeta \, \left( \frac{n+1}{2}\right)\left( \frac{|w|}{4\pi}\right)^{(n+1)/2},
\end{equation*}
which implies \eqref{eq:5.3}.
\end{proof}
In the sequel throughout we will denote by $q\in (0,1)$ a constant which may be different at each occurrence.
\begin{lemma} \label{lem5.2}
Suppose that the sequences $(a_n), (b_n), (c_n)$ are given by
\begin{flalign}
\label{eq:5.4}
a_n & := \frac{n+1}{2\log 2} - \frac{1}{2},  \qquad b_n := \frac{1 - \log 2}{(2\log 2)^2}\, (n+1) - \frac{1}{4}, \\*[0.3cm]
\label{eq:5.5}
c_n & := \frac{\sqrt{n} (2 - 3\log 2)}{(2\sqrt{b_n} \log 2)^3}\, (n+1),
\end{flalign}
then we have
\begin{align}
\label{eq:5.6}
\mu_n & = a_n + {\mathscr O} (q^n), \qquad \sigma_n^2 = b_n + {\mathscr O} (q^n), \\
\label{eq:5.7}
\lambda_{3,n} & = c_n + {\mathscr O} (q^n), \qquad \text{\rm as } n\to\infty.
\end{align}
\end{lemma}
\begin{proof}
Using standard formulae from probability \cite{29} and in particular \eqref{eq:4.8} we get
\begin{flalign}
\label{eq:5.8}
\mu_n & = \frac{Q_n'(1)}{Q_n (1)}, \quad
       \sigma_n^2 = \frac{Q_n'' (1)}{Q_n (1)} + \frac{Q_n' (1)}{Q_n (1)} - \left( \frac{Q_n'(1)}{Q_n(1)}\right)^2,  \\*[0.3cm]
\label{eq:5.9}
\lambda_{3,n} & = \frac{\sqrt{n}}{\sigma_n^3} 
\Bigg\{ 
\frac{Q_n'''(1)}{Q_n (1)} + 3\, \frac{Q_n''(1)}{Q_n (1)} - 3\, \frac{Q_n'' (1)}{Q_n (1)} \, \frac{Q_n'(1)}{Q_n(1)} 
              + \frac{Q_n'(1)}{Q_n (1)} - 3\, \left( \frac{Q_n'(1)}{Q_n(1)}\right)^2 
               + 2\, \left( \frac{Q_n' (1)}{Q_n (1)}\right)^3 \Bigg\}.               
\end{flalign}
Now based on the representation (3.13) for $Q_n$ in Lemma \ref{lem3.6} ii) and with the help of Lemma \ref{lem5.1}, taken at $w = \log 2$, straightforward computations lead to the approximations \eqref{eq:5.6} and \eqref{eq:5.7}.
\end{proof}
Now our main result for the Stirling numbers is stated in
\begin{theorem} \label{theo5.3}
Suppose that the sequences $(a_n), (b_n), (c_n)$ are given by \eqref{eq:5.4}, \eqref{eq:5.5}, then we have
\begin{equation} \label{eq:5.10}
\frac{2\sqrt{b_n} j! (\log 2)^{n+1}}{n!}\, 
\left\{ \begin{array}{c} n \\ j \end{array} \right\} =
\frac{1}{\sqrt{2\pi}}~ e^{-x^2/2} \, 
\left( 1 + \frac{c_n (x^3 - 3x)}{6\sqrt{n}}\right) + 
o \, \left( \frac{1}{\sqrt{n}}\right),
\end{equation}
as $n\to\infty$, uniformly in $j\in\Z$, where $x = (j-a_n)/\sqrt{b_n}$.
\end{theorem}
\begin{proof}
According to our preparations we may apply the basic Lemma \ref{lem4.1} with $k = 3$ to the probabilities \eqref{eq:5.2} which is permitted, because condition \eqref{eq:4.7} is satisfied in view of Lemma \ref{lem5.2}. Hence, observing \eqref{eq:4.9}, \eqref{eq:4.10}, we get
\begin{equation} \label{eq:5.11}
\frac{\sigma_n j! \left\{ \begin{array}{c} n \\ j \end{array}\right\}}{Q_n (1)} = \frac{1}{\sqrt{2\pi}}~ e^{-y^2/2}\,
\left( 1 + \frac{\lambda_{3,n} (y^3 - 3y)}{6\sqrt{n}}\right) + o\, \left( \frac{1}{\sqrt{n}}\right),
\end{equation}
as $n\to\infty$, uniformly in $j\in\Z$, where $y = (j-\mu_n)/\sigma_n$, and $\mu_n, \sigma_n, \lambda_{3,n}$ are taken from Lemma \ref{lem5.2}. Next, we emphasize that the latter quantities are approximated by $a_n, b_n, c_n$ respectively holding at a geometric rate each. Also, by Lemma \ref{lem5.1} and Lemma \ref{lem3.6} ii) we have
\[
Q_n (1) = \frac{n!}{2(\log 2)^{n+1}} \, \big( 1 + {\mathscr O} (q^n)\big), \quad \text{as } n \to\infty\, .
\]
Then, using these approximation properties, somewhat tedious but straightforward computations show that \eqref{eq:5.11} implies \eqref{eq:5.10}. We omit the elementary details.
\end{proof}
We conclude this section by some comments on Theorem \ref{theo5.3}. Although for the general local central limit theorem the error term in \eqref{eq:4.11} holds uniformly in $j\in\Z$, the most valuable information is provided for $j$ such that the quantity $x$ in \eqref{eq:4.11} remains bounded. The same statement is true for \eqref{eq:5.10}, that is for $j$ such that
$x = (j-a_n)/\sqrt{b_n}$ is bounded. If $j$ is close to $a_n \sim n/2\log 2$, then we may expect a good approximation of
$\big\{ ^n_j\big\}$ by the quantity
\begin{equation} \label{eq:5.12}
A(n,j) := \frac{n!}{2\sqrt{2\pi b_n} j! (\log 2)^{n+1}}~ e^{-x^2/2}\,
\left( 1 + \frac{c_n (x^3 - 3x)}{6\sqrt{n}}\right).
\end{equation}
For illustration we choose $n = 500$ and $j = 360$ which implies a good relative comparison by
\begin{equation} \label{eq:5.13}
\frac{\left\{ \begin{array}{c} n \\ j \end{array}\right\}}{A(n,j)} = 0.999376\ldots
\end{equation}
Although in contrast to most of the asymptotic expressions for the Stirling numbers in the literature \cite{4}, \cite{5}, \cite{20},
 \cite{36}, \cite{39} our result contains explicit and elementary quantities only, we emphasize that \eqref{eq:5.10} provides an effective approximation of $\big\{ ^n_j\big\}$ when $j$ is close to $a_n$ as described above. As surveyed in \cite{36} many of the known asymptotics only are useful for other regions of $j$. So far Theorem \ref{theo5.3} gives a supplement of these asymptotics. \\

Finally we mention the asymptotic normality of the Stirling partition distribution itself, which is defined by
\[
p(n,j) = \frac{\left\{ \begin{array}{c} n \\ j\end{array}\right\}}{B_n}\, , \quad 0 \leq j \leq n,
\]
$B_n = \sum_{j=0}^n \big\{ ^n_j\big\}$ being the Bell numbers with mean $\mu_n \sim n/\log n$ and variance
$\sigma_n^2 \sim n/\log^2n,$ $n \to\infty$, which are different from the quantities in \eqref{eq:5.6} 
\cite[p. 692]{14}, \cite{19}, \cite[p. 115]{35}.
\section{Asymptotics for the Chebyshev-Stirling numbers}
In this section we prove our second result being an analogue of Theorem \ref{theo5.3} for the Chebyshev-Stirling numbers. To this end again we use the probabilistic tools of section 4 now with $A_n (s) = L_n (s)$ and
\begin{equation} \label{eq:6.1}
p(n,j) := \frac{(2j)! \left\{ \begin{array}{c} n \\ j \end{array}\right\}_{1/2}}{L_n (1)},
\end{equation}
where $L_n (s)$ is defined in \eqref{eq:3.6}. This is possible in view of Lemma \ref{lem3.4}. \\

Again we start by computing and approximating moments and cumulants for the distribution in \eqref{eq:6.1}. Looking at \eqref{eq:3.9} we introduce the number
\begin{equation} \label{eq:6.2}
\omega := 2\log\, \frac{\sqrt{5}+1}{2}
\end{equation}
being the unique positive solution of $2(\cosh w - 1) = 1$.
\begin{lemma} \label{lem6.1}
Suppose that the sequences $(a_n'), (b_n'), (c_n')$ are given by
\begin{flalign}
\label{eq:6.3} a_n' & := \frac{2n+1}{\sqrt{5}\omega} - \frac{2}{5}, \qquad
b_n'   := \left( \frac{1}{5\omega^2} - \frac{2}{5\sqrt{5}\omega}\right)\, (2n+1) - \frac{2}{25}, \\*[0.3cm]
\label{eq:6.4} c_n' & := \frac{2\sqrt{n}}{(5\sqrt{b_n'} \omega)^3} \, 
          \Big\{ \big( 2\sqrt{5} \omega^2 - 30\omega + 10\sqrt{5}\big)\, n + 3\omega^3 + \sqrt{5}\,\omega^2 -
          15\omega + 5\sqrt{5}\Big\},
\end{flalign}
then we have
\begin{align}
\label{eq:6.5}\mu_n & = a_n' + {\mathscr O}(q^n), \quad \sigma_n^2 = b_n' + {\mathscr O} (q^n), \\
\label{eq:6.6}\lambda_{3,n} & = c_n' + {\mathscr O} (q^n), \quad \text{\rm as } n \to\infty.
\end{align}
\end{lemma}
\begin{proof}
We employ the same arguments used in the proof of Lemma \ref{lem5.2}. Also now relations \eqref{eq:5.8} and \eqref{eq:5.9} are valid with $Q_n$ replaced by $L_n$. The representation \eqref{eq:3.9} in Lemma \ref{lem3.3} iii) together with Lemma \ref{lem5.1}, taken at $w = \omega$, implies the approximations \eqref{eq:6.5} and \eqref{eq:6.6} by straightforward calculations.
\end{proof}
The main result for the Chebyshev-Stirling numbers is furnished by
\begin{theorem} \label{theo6.2}
Suppose that the sequences $(a_n'), (b_n'), (c_n')$, and the number $\omega$ are given by \eqref{eq:6.3}, \eqref{eq:6.4}, and \eqref{eq:6.2} respectively, then we have
\begin{equation} \label{eq:6.7}
\frac{\sqrt{5b_n'} (2j)! \omega^{2n+1}}{2(2n)!} \,
\left\{ \begin{array}{c} n \\ j \end{array} \right\}_{1/2} = \frac{1}{\sqrt{2\pi}}~ e^{-x^2/2}~
\left( 1 + \frac{c_n' (x^3 - 3x)}{6\sqrt{n}}\right) 
+ o\, \left( \frac{1}{\sqrt{n}}\right),
\end{equation}
as $n\to\infty$, uniformly in $j\in\Z$, where $x = (j-a_n')/\sqrt{b_n'}$.
\end{theorem}
\begin{proof}
This will be accomplished as that one of Theorem \ref{theo5.3}, now by an application of Lemma \ref{lem4.1} with $k = 3$ to the probabilities \eqref{eq:6.1}. We note that the crucial condition \eqref{eq:4.7} is satisfied in view of Lemma \ref{lem6.1}. It follows that
\begin{equation} \label{eq:6.8}
\frac{\sigma_n (2j)!}{L_n (1)} ~ \left\{ \begin{array}{c} n \\ j \end{array} \right\}_{1/2} =
\frac{1}{\sqrt{2\pi}}~ e^{-y^2/2} \, 
\left( 1 + \frac{\lambda_{3,n} (y^3 - 3y)}{6\sqrt{n}}\right) 
+ o\, \left( \frac{1}{\sqrt{n}}\right),
\end{equation}
as $n\to\infty$, uniformly in $j\in\Z$, where $y = (j - \mu_n)/\sigma_n$, and $\mu_n, \sigma_n,\lambda_{3,n}$ are taken from Lemma \ref{lem6.1}. Finally, after some routine calculations we arrive at \eqref{eq:6.7} with the help of Lemmas 
\ref{lem6.1}, \ref{lem5.1}, \ref{lem3.6} ii) and the formula
\[
L_n (1) = \frac{2(2n)!}{\sqrt{5} \omega^{2n+1}} \, \big( 1 + {\mathscr O} (q^n)\big), \quad \text{as } n\to\infty.
\]
\end{proof}
The comments on Theorem \ref{theo5.3} essentially apply to Theorem \ref{theo6.2} likewise. Similarly we may expect good approximations of the Chebyshev-Stirling numbers via \eqref{eq:6.7}, if $x = (j-a_n')/\sqrt{b_n'}$ remains bounded, in particular when $j$ is close to $a_n' \sim 2n/\sqrt{5}\omega$. For illustration we put
\begin{equation} \label{eq:6.9}
A' (n,j) := \frac{2(2n)!}{\sqrt{10\pi b_n'} (2j)! \omega^{2n+1}}~ e^{-x^2/2} \,
\left( 1 + \frac{c_n' (x^3 - 3x)}{6\sqrt{n}}\right)
\end{equation}
and choose $n = 500, j = 461$ to obtain the relative comparison
\begin{equation} \label{eq:6.10}
\frac{\left\{ \begin{array}{c} n \\ j \end{array}\right\}_{1/2}}{A'(n,j)}= 1.000891\ldots
\end{equation}
Also we mention that one could improve both results in Theorems \ref{theo5.3} and \ref{theo6.2} by making use of higher terms in Lemma \ref{lem4.1}. However, we will not pursue this topic. \\

Finally, we conclude this section by the statement that the numbers $(2j)! \big\{ ^n_j\}_{1/2}$ are asymptotically normal. Since a proof makes use of routine arguments on the basis of either Lyapunov's theorem \cite[p. 23]{35} or of our Theorem \ref{theo6.2} directly, we omit detailed explanations.
\begin{theorem} \label{theo6.3}
Suppose that the sequences $(a_n'), (b_n')$, and the number $\omega$ are given by \eqref{eq:6.3} and \eqref{eq:6.2} respectively, then for all $y\in\R$ we have
\[
\lim_{n\to\infty} \, \frac{\sqrt{5} \omega^{2n+1}}{2(2n)!} \, 
\sum_{j\leq a_n' + y\sqrt{b_n'}} (2j)! \,
\left\{ \begin{array}{c} n \\ j \end{array} \right\}_{1/2} = 
\frac{1}{\sqrt{2\pi}}\, \int\limits_{-\infty}^y e^{-t^2/2}\, dt.
\]
\end{theorem}

$^{a}$ Department of Mathematics, University of Trier, 54286 Trier, Germany \\ 
\textit{Email address:} {\tt gawron@uni-trier.de}\\

$^{b,*}$ Department of Mathematics, Baylor University, Waco, TX 76798-7328, USA \\
\textit{Email address:} {\tt lance\_littlejohn@baylor.edu}\\
$^{*}$Corresponding Author\\

$^{c,**}$ Department of Mathematics, Katholieke Universiteit Leuven, Celestijnenlaan 200B, B-3001 Leuven, Belgium\\
\textit{Email address:} {\tt thorsten.neuschel@wis.kuleuven.be}\\
$^{**}$ Thorsten Neuschel gratefully acknowledges support from KU Leuven research grant OT/12/073.\\

\end{document}